\documentclass[a4paper,12pt]{article}
\usepackage{amsmath,amsthm,amssymb}

\newtheorem{thm}{Theorem}
\newtheorem{lem}[thm]{Lemma}
\newtheorem{prop}[thm]{Proposition}

\theoremstyle{definition}
\newtheorem{defin}[thm]{Definition}

\newcommand{\C}{\mathbb{C}}
\newcommand{\R}{\mathbb{R}}
\newcommand{\Span}{\mathop{\mathrm{span}}\nolimits}

\begin{document}

\begin{center}
\Large{\textbf{
Nonlinear instability of linearly unstable standing waves for 
nonlinear Schr\"odinger equations}}
\end{center}

\vspace{5mm}

\begin{center}
{\large Vladimir Georgiev} $^{1}$ and {\large Masahito Ohta} $^{2}$
\end{center}
\begin{center}
$^{1}$ Dipartimento di Matematica, Universit\`a degli Studi di Pisa, \\
Largo Bruno Pontecorvo 5, I-56127 Pisa, Italy \\
georgiev@dm.unipi.it \\
$^{2}$ Department of Mathematics, Saitama University, \\
Saitama 338-8570, Japan \\
mohta@mail.saitama-u.ac.jp
\end{center}

\begin{abstract}
We study the instability of standing waves for
nonlinear Schr\"odinger equations.
Under a general assumption on nonlinearity,
we prove that linear instability implies orbital instability
in any dimension.
For that purpose, we establish a Strichartz type estimate
for the propagator generated by the linearized operator
around standing wave.
\end{abstract}

\section{Introduction}\label{sect:intro}

In this paper we study the instability of standing waves
for nonlinear Schr\"odinger equations
\begin{equation}\label{nls}
i\partial_tu+\Delta u+g(|u|^2)u=0, \quad (t,x)\in \R\times \R^N,
\end{equation}
where $u$ is a complex-valued function of $(t,x)$,
and $g$ is a real-valued function.
A typical example of nonlinearity is
$g(|u|^2)u=|u|^{p-1}u$ with $1<p<2^*-1$,
where $2^*=2N/(N-2)$ if $N\ge 3$ and $2^*=\infty$ if $N=1,2$.
Precise assumptions on the nonlinearity will be made later.
By a standing wave we mean a solution of \eqref{nls}
of the form $u(t,x)=e^{i\omega t}\varphi(x)$,
where $\omega\in \R$ and $\varphi\in H^1(\R^N)\setminus\{0\}$
is a solution of the stationary problem
\begin{equation}\label{sp}
-\Delta \varphi+\omega \varphi-g(|\varphi|^2)\varphi=0,\quad x\in \R^N.
\end{equation}

For the special case $g(|u|^2)u=|u|^{p-1}u$ with $1<p<2^*-1$,
the following results are well-known.
For each $\omega>0$,
the stationary problem \eqref{sp} has a unique positive radial solution
in $H^1(\R^N)$ (see \cite{strauss,BL} for existence,
and \cite{kwo} for uniqueness).
We call it ground state.
When $N\ge 2$, other than the ground state,
there exist infinitely many solutions of \eqref{sp} in $H^1(\R^N)$.
We call them excited states.
For the ground state $\varphi$ of \eqref{sp} with $\omega>0$,
the standing wave $e^{i\omega t}\varphi$ is
orbitally stable if $1<p<1+4/N$,
while it is orbitally unstable if $1+4/N\le p<2^*-1$
(see \cite{BC,CL,wei1}).
For more general nonlinearity,
Shatah and Strauss \cite{SS1} gave
a general condition for orbital instability
of ground state-standing waves for \eqref{nls}
constructing suitable Lyapunov functionals
(see also \cite{GSS1} and \cite{gon, mae, oh, oht}).
We remark that these results are mostly limited to
ground states and are not applicable to excited states.
Here, we recall the definition of
orbital stability and instability of standing waves.

\begin{defin}
We say that the standing wave $e^{i\omega t}\varphi$
is \emph{orbitally stable}
if for any $\varepsilon>0$ there exists $\delta>0$ such that
if $u_0\in H^1(\R^N)$ and $\|u_0-\varphi\|_{H^1}<\delta$,
then the solution $u(t)$ of \eqref{nls} with $u(0)=u_0$
exists globally and satisfies
$$\inf_{(\theta,y)\in \R\times \R^N}
\|u(t)-e^{i\theta}\varphi(\cdot+y)\|_{H^1}<\varepsilon$$
for all $t\ge 0$.
Otherwise, $e^{i\omega t}\varphi$ is called \emph{orbitally unstable}
or \emph{nonlinearly unstable}.
\end{defin}

While, $e^{i\omega t}\varphi$ is said to be \emph{linearly unstable}
if the linearized operator $A=JH$ around the standing wave
has an eigenvalue with positive real part
(for the definition of $J$ and $H$,
see \eqref{nlsR} and \eqref{def:H} below).
The linear instability of standing waves for \eqref{nls} was studied by
Jones \cite{jon} and Grillakis \cite{gri1, gri2}
(see also \cite{GSS2, miz1,miz3}).
In particular,
for the case $g(|u|^2)u=|u|^{p-1}u$ with $1+4/N<p<2^*-1$,
it is proved in \cite{gri1} that
for any radially symmetric, real-valued solution $\varphi$
of \eqref{sp} with $\omega>0$,
$e^{i\omega t}\varphi$ is linearly unstable.
The result in \cite{gri1} guarantees
that among radially symmetric solutions,
one can find oscillating solutions (i.e. solutions changing the sign)
and these solutions shall generate excited states $e^{i\omega t}\varphi$.
On the other hand, Mizumachi \cite{miz1,miz3}
considered complex-valued solutions of \eqref{sp}
in $\R^2$ of the form $\varphi_m(x)=e^{im\theta}\phi(r)$,
where $m$ is a positive integer,
and $r$, $\theta$ are the polar coordinates in $\R^2$
(see \cite{IW,lio} for existence of $\varphi_m$).
It is proved that if $p>3$ then for any $m$,
$e^{i\omega t}\varphi_m$ is linearly unstable (\cite{miz1}),
and that if $1<p<3$ then for sufficiently large $m$,
$e^{i\omega t}\varphi_m$ is linearly unstable (\cite{miz3}).

However, it is a highly nontrivial problem whether linear instability 
implies orbital instability for \eqref{nls}, 
especially in higher dimensional case (see \cite{dB, DiMG, miz2, SS2}).
Even in two dimensional case, some technical difficulties arise
from the estimates of nonlinear terms
(see Lemma 13 of \cite{CCO1}).
For the case $N\le 3$,
a satisfactory answer for this problem was given by
Colin, Colin and Ohta \cite{CCO2}.
The main idea in \cite{CCO2} is to employ time derivative
in the estimates of nonlinear terms
without using space derivatives directly,
and to apply the $H^2$-regularity of $H^1$-solutions for \eqref{nls}.
However, the proof of \cite{CCO2} is based on the $L^2$-estimate
on the propagator $e^{tA}$ generated by the linearized operator $A$,
and the restriction $N\le 3$ comes from the embedding
$H^2(\R^N)\hookrightarrow L^{\infty}(\R^N)$.

The main goal of this work is to show that linear instability
implies orbital instability for \eqref{nls} in any dimension $N\ge 1$
(see Theorem \ref{mainthm} below).
In particular, for the case $g(|u|^2)u=|u|^{p-1}u$ with $1+4/N<p<2^*-1$,
it follows from the linear instability result of \cite{gri1}
and our Theorem \ref{mainthm} that
for any radially symmetric, real-valued solution $\varphi$
of \eqref{sp} with $\omega>0$, $e^{i\omega t}\varphi$ is 
orbitally unstable in any dimension. 

Our approach is based on appropriate Strichartz type estimate
for the propagator $e^{tA}$
and gives the possibilities for further generalization.
We have chosen the model of the nonlinear Schr\"odinger equation \eqref{nls}  for simplicity, but even in this case
one needs to apply spectral mapping result $\sigma (e^A)=e^{\sigma (A)}$
discussed in the work of
Gesztesy, Jones, Latushkin and Stanislavova \cite{GJLS}.
If one considers complex-valued solutions of \eqref{sp},
then the assertion
$$ \text{linear instability}   \Longrightarrow
\text{orbital  instability}$$
depends on the possible generalization
of the property $\sigma (e^A)=e^{\sigma (A)}$ for 
the linearized operator $A$ around complex-valued excited states.
Since our goal is to give general argument working for complex-valued
excited states as well,
we have to make suitable generalization of the result in \cite{GJLS}
(see Section \ref{sect:SMT}).

Here, we give an outline of the paper more precisely.
In what follows, we often identify $z\in \C$ with
${}^{t}(\Re z,\Im z)\in \R^2$,
and write $z={}^{t}(\Re z,\Im z)$.
We define $f(z)=-g(|z|^2)z$ for $z\in \R^2$.
Then, \eqref{nls} is rewritten as
\begin{equation}\label{nlsR}
\partial_tu=J(-\Delta u+f(u)), \quad
J=\left[\begin{array}{cc}
0 & 1 \\
-1 & 0
\end{array}\right], \quad
u=\left[\begin{array}{c}
\Re u \\
\Im u
\end{array}\right].
\end{equation}
We assume that $f\in C^1(\R^2,\R^2)$,
and denote the derivative of $f$ at $z\in \R^2$
by $Df(z)$, which is a $2\times 2$-real symmetric matrix
and is given by
\begin{equation}\label{DF}
Df(z)=-\left[\begin{array}{cc}
2g'(|z|^2)(\Re z)^2+g(|z|^2) & 2g'(|z|^2)\Re z \Im z \\
2g'(|z|^2)\Re z \Im z  & 2g'(|z|^2)(\Im z)^2+g(|z|^2)
\end{array}\right].
\end{equation}
For nonlinearity, we assume the following.

\vspace{2mm}
\noindent \textbf{(H1)} \hspace{1mm}
$g$ is a real-valued continuous function on $[0,\infty)$,
and $f(z)=-g(|z|^2)z$ is decomposed
as $f=f_1+f_2$ with $f_j\in C^1(\R^2,\R^2)$,
$f_j(0)=0$, $Df_j(0)=O$, $j=1,2$,
and there exist constants $C$ and $1<p_j<2^*-1$ such that
\begin{equation}\label{ass:f}
|Df_j(z_1)-Df_j(z_2)|
\le C \left\{\begin{array}{ll}
|z_1-z_2|^{p_j-1} &
\hspace{2mm} \mbox{if} \hspace{3mm}
1<p_j\le 2 \\
(|z_1|^{p_j-2}+|z_2|^{p_j-2})|z_1-z_2| &
\hspace{2mm} \mbox{if} \hspace{3mm}
p_j>2
\end{array}\right.
\end{equation}
for all $z_1$, $z_2\in \R^2$.
\vspace{2mm}

Remark that the typical example $f(z)=-|z|^{p-1}z$
satisfies (H1) for $1<p<2^*-1$
(see Lemma 2.4 of \cite{GV}).
Moreover, the Cauchy problem for \eqref{nls}
is locally well-posed in $H^1(\R^N)$
(see \cite{kat} and \cite[Chapter 4]{caz}).

For a solution of \eqref{sp}, we assume the following.

\vspace{2mm}
\noindent \textbf{(H2)} \hspace{1mm}
$\omega>0$ is a constant and $\varphi\in H^1(\R^N)$ is
a complex-valued nontrivial solution of \eqref{sp}.
\vspace{2mm}

For the existence of solutions of \eqref{sp},
see, e.g., \cite{BL,IW,lio,strauss}.
By the elliptic regularity theory,
we see that $\varphi\in H^2(\R^N)\cap C^2(\R^N)$
and $\varphi(x)$ decays to $0$ exponentially as $|x|\to \infty$.
Remark that we consider not only real-valued solutions of \eqref{sp}
but also complex-valued solutions, and that by \eqref{DF},
$Df(\varphi)$ is a diagonal matrix if $\varphi$ is real-valued,
but not in general.

By a change of variables $u(t)=e^{i\omega t}\left(\varphi+v(t)\right)$
in \eqref{nls} or \eqref{nlsR}, we have
\begin{equation}\label{eq:v}
\partial_tv=Av+h(v),
\end{equation}
where $v={}^{t}(\Re v,\Im v)$,
$A=JH$, $h(v)=
J[f(\varphi+v)-f(\varphi)-Df(\varphi)v]$, and
\begin{equation}\label{def:H}
H=H_0+Df(\varphi), \quad
H_0=\left[\begin{array}{cc}
-\Delta+\omega & 0 \\
0 & -\Delta+\omega
\end{array}\right].
\end{equation}

For the linearized operator $A=JH$,
we assume the following.

\vspace{2mm}
\noindent \textbf{(H3)} \hspace{1mm}
The operator $A$ has an eigenvalue $\lambda_0$
such that $\Re \lambda_0>0$.
\vspace{2mm}

As stated above,
sufficient conditions for (H3) are studied by
\cite{gri1, gri2, GSS2, jon, miz1, miz3}.
See also \cite{CGNT, cuc, CPV, wei2} for spectral properties of $A$.
We now state the main result of this paper.

\begin{thm}\label{mainthm}
Assume $({\rm H1})$--$({\rm H3})$.
Then, the standing wave $e^{i\omega t}\varphi$
of \eqref{nls} is orbitally unstable.
\end{thm}

The rest of the paper is organized as follows.
In Section \ref{sect:str}, assuming that the propagator $e^{tA}$
satisfies an exponential growth condition \eqref{growth},
we introduce a suitable norm \eqref{norm}
and establish a Strichartz type estimate for $e^{tA}$.
In Section \ref{sect:instability}, we prove Theorem \ref{mainthm}.
In the proof, we apply the Strichartz type estimate for $e^{tA}$
proved in Section \ref{sect:str},
and we employ time derivative instead of space derivatives
in the estimates of nonlinear terms as in \cite{CCO2}.
Finally, in Section \ref{sect:SMT},
we give some remarks on the spectral mapping theorem for $e^A$
due to Gesztesy, Jones, Latushkin and Stanislavova \cite{GJLS}.

\section{Strichartz estimates}\label{sect:str}

Let $V_{jk}\in L^{\infty}(\R^N,\R)$ for $j,k=1,2$,
and we consider linear operators
\begin{equation}\label{AV}
A=A_0+V, \quad A_0=JH_0, \quad
V=\left[\begin{array}{cc}
V_{11} & V_{12} \\
V_{21} & V_{22}
\end{array}
\right]
\end{equation}
on $L^2(\R^N)\times L^2(\R^N)$
with domains $D(A_0)=D(A)=H^2(\R^N)\times H^2(\R^N)$,
where $J$ and $H_0$ are defined in \eqref{nlsR} and \eqref{def:H}.
Let $e^{tA_0}$ and $e^{tA}$ be the strongly continuous groups
on $L^2(\R^N)\times L^2(\R^N)$
generated by $A_0$ and $A$ respectively, and we define
\begin{align*}
\Gamma_0[f](t)=\int_{0}^{t}e^{(t-s)A_0}f(s)\,ds, \quad
\Gamma [f](t)=\int_{0}^{t}e^{(t-s)A}f(s)\,ds.
\end{align*}
Moreover, we denote $L^r:=L^r(\R^N)\times L^r(\R^N)$ and
$L_T^qY:=L^q((0,T),Y)$ for a Banach space $Y$.
Note that $u(t)=e^{tA}\psi+\Gamma [f](t)$ satisfies
\begin{equation}\label{eq:u}
\partial_tu=Au+f(t)=A_0u+Vu+f(t), \quad u(0)=\psi,
\end{equation}
and $u_0(t)=e^{tA_0}\psi+\Gamma_0[f](t)$ satisfies
\begin{equation}\label{eq:u0}
\partial_tu_0=A_0u_0+f(t)=Au_0+f(t)-Vu_0, \quad u_0(0)=\psi.
\end{equation}
We assume that there exist positive constants $C$ and $\nu$
such that
\begin{equation}\label{growth}
\|e^{tA}\|_{B(L^2)}\le Ce^{\nu t}
\end{equation}
for all $t\ge 0$.
For $\lambda>0$, we define functions
$e^{+}_{\lambda}$ and $e^{-}_{\lambda}$ by
$e^{\pm}_{\lambda}(t)=e^{\pm \lambda t}$ for $t\in \R$.
Moreover, we define
\begin{equation}\label{norm}
\|f\|_{L^{q,\lambda}_TY}
:=e^{\lambda T}\|e^{-}_{\lambda}f\|_{L_T^qY}.
\end{equation}
Note that $\|f\|_{L_T^qY}\le \|f\|_{L^{q,\lambda}_TY}
\le \|f\|_{L^{q,\mu}_TY}$
for $0<\lambda<\mu$ and $T>0$.
The H\"older conjugate of $q$ is denoted by $q'$.
For the definition of admissible pairs
and the standard Strichartz estimates for $e^{it\Delta}$,
see, e.g., \cite[Section 2.3]{caz}.

\begin{lem}\label{lem1}
Assume $V\in L^{\infty}(\R^N)$ and \eqref{growth}.
Let $0<\nu<\mu$ and let $(q,r)$ be any admissible pair.
Then, there exists a constant $C$
independent of $\psi$, $f$ and $T$ such that
$u(t)=e^{tA}\psi+\Gamma [f](t)$ satisfies
$$\|u(t)\|_{L^2}
\le C\left(e^{\nu t}\|\psi\|_{L^2}
+e^{\mu t}\|e^{-}_{\mu}f\|_{L^{q'}_TL^{r'}}\right)$$
for all $t\in [0,T]$.
\end{lem}

\begin{proof}
Let $u_0(t)=e^{tA_0}\psi+\Gamma_0[f](t)$.
Then, by \eqref{eq:u} and \eqref{eq:u0}, we have
$$\partial_t(u-u_0)=A(u-u_0)+Vu_0, \quad (u-u_0)(0)=0,$$
so $u-u_0=\Gamma [Vu_0]$.
By the assumption \eqref{growth}, we have
\begin{align*}
\|u(t)-u_0(t)\|_{L^2}
&\le \int_{0}^{t}\|e^{(t-s)A}Vu_0(s)\|_{L^2}\,ds \\
&\le C\|V\|_{L^{\infty}}\int_{0}^{t}e^{\nu (t-s)}\|u_0(s)\|_{L^2}\,ds
\end{align*}
for all $t\in [0,T]$.
Here, by the standard Strichartz estimate for $e^{it\Delta}$, we have
$$\|u_0(t)\|_{L^2}\le C(\|\psi\|_{L^2}+\|f\|_{L^{q'}_tL^{r'}})
\le C(\|\psi\|_{L^2}+e^{\mu t}\|e^{-}_{\mu}f\|_{L^{q'}_TL^{r'}})$$
for all $t\in [0,T]$. Thus,
\begin{align*}
&\|u(t)\|_{L^2}
\le \|u_0(t)\|_{L^2}+\|u(t)-u_0(t)\|_{L^2} \\
&\le \|u_0(t)\|_{L^2}+C\int_{0}^{t}e^{\nu (t-s)}\|\psi\|_{L^2}\,ds
+Ce^{\nu t}\int_{0}^{t}e^{(\mu-\nu)s}
\|e^{-}_{\mu}f\|_{L^{q'}_TL^{r'}}\,ds \\
&\le C(e^{\nu t}\|\psi\|_{L^2}
+e^{\mu t}\|e^{-}_{\mu}f\|_{L^{q'}_TL^{r'}})
\end{align*}
for all $t\in [0,T]$. This completes the proof.
\end{proof}

\begin{prop}\label{prop2}
Assume $V\in L^{\infty}(\R^N)$ and \eqref{growth}.
Let $0<\lambda<\nu<\mu$,
and let $(q_1,r_1)$ and $(q_2,r_2)$ be any admissible pairs.
Then, there exists a constant $C$
independent of $\psi$, $f$ and $T$ such that
$u(t)=e^{tA}\psi+\Gamma [f](t)$ satisfies
$$\|u\|_{L^{q_1,\lambda}_TL^{r_1}}
\le C\left(e^{\nu T}\|\psi\|_{L^2}
+\|f\|_{L^{q_2',\mu}_TL^{r_2'}}\right).$$
\end{prop}

\begin{proof}
We put $v(t)=e^{-\lambda t}u(t)$.
Then, by \eqref{eq:u}, we have
$$\partial_tv=A_0v+(V-\lambda)v+e^{-\lambda t}f(t),
\quad v(0)=\psi.$$
By the standard Strichartz estimate for $e^{it\Delta}$, we have
$$\|e^{-}_{\lambda}u\|_{L^{q_1}_TL^{r_1}}=\|v\|_{L^{q_1}_TL^{r_1}}
\le C(\|\psi\|_{L^2}+\|(V-\lambda)v\|_{L^1_TL^2}
+\|e^{-}_{\lambda}f\|_{L^{q_2'}_TL^{r_2'}}).$$
Here, by Lemma \ref{lem1}, we have
\begin{align*}
\|(V-\lambda)v\|_{L^1_TL^2}
&\le (\|V\|_{L^{\infty}}+\lambda)\|v\|_{L^1_TL^2}
\le C\int_{0}^{T}e^{-\lambda t}\|u(t)\|_{L^2}\,dt \\
&\le C\int_{0}^{T}\{e^{(\nu-\lambda)t}\|\psi\|_{L^2}
+e^{(\mu-\lambda)t}\|e^{-}_{\mu}f\|_{L^{q_2'}_TL^{r_2'}}\}\,dt \\
&\le C\{e^{(\nu-\lambda)T}\|\psi\|_{L^2}
+e^{(\mu-\lambda)T}\|e^{-}_{\mu}f\|_{L^{q_2'}_TL^{r_2'}}\}.
\end{align*}
Moreover, since $\|e^{-}_{\lambda}f\|_{L^{q_2'}_TL^{r_2'}}
\le e^{(\mu-\lambda)T}\|e^{-}_{\mu}f\|_{L^{q_2'}_TL^{r_2'}}$,
we obtain the desired estimate.
\end{proof}

\section{Proof of Theorem \ref{mainthm}}\label{sect:instability}

In this section we assume (H1)--(H3),
and prove Theorem \ref{mainthm}.
For $j=1,2$, we put
$$h_j(v)=J[f_j(\varphi+v)-f_j(\varphi)-Df_j(\varphi)v],
\quad r_j=p_j+1,$$
and let $(q_j,r_j)$ be the corresponding admissible pair.
Note that $h(v)=h_1(v)+h_2(v)$ in \eqref{eq:v}.

\begin{lem}\label{lem:chi}
There exist $\lambda^*\in \C$ and
$\chi\in H^2(\R^N,\C)^2$
such that $\Re \lambda^*>0$,
$A\chi=\lambda^* \chi$ and $\|\chi\|_{L^2}=1$.
Moreover, $e^{tA}$ satisfies \eqref{growth}
for some $\nu$ with
$\Re \lambda^*<\nu<(1+\alpha)\Re \lambda^*$,
where
\begin{equation}\label{alp}
\alpha:=\min\{1,r_1-2,r_2-2\}.
\end{equation}
\end{lem}

\begin{proof}
Since $Df(\varphi)$ decays exponentially at infinity,
Weyl's essential spectrum theorem implies that
$\sigma_{{\rm ess}}(A)\subset \{z\in \C:\Re z=0\}$.
Moreover, the number of eigenvalues of $A=JH$
in $\{z\in \C:\Re z>0\}$ is finite
(see, e.g., Theorem 5.8 of \cite{GSS2}).
Therefore, by (H3), there exists an eigenvalue $\lambda^*$ of $A$
such that $\Re \lambda^*=\max \{\Re z: z\in \sigma (A)\}>0$.
Further, by the spectral mapping theorem due to
Gesztesy, Jones, Latushkin and Stanislavova \cite{GJLS},
we have $\sigma (e^A)=e^{\sigma (A)}$.
Here we need some modification of \cite{GJLS}
when $\varphi$ is not real-valued.
We shall discuss it in Section \ref{sect:SMT}.
Then, the spectral radius of $e^{A}$ is $e^{\Re \lambda^*}$.
Finally, by Lemma 3 of \cite{SS2}, we see that
$e^{tA}$ satisfies \eqref{growth} for some $\nu$ with
$\Re \lambda^*<\nu<(1+\alpha)\Re \lambda^*$.
\end{proof}

\begin{lem}\label{lem21}
There exists a constant $C$ such that
$$\|h_j(v)\|_{L^2}+\|h_j(v)\|_{L^{r_j'}}
\le C\left(\|v\|_{H^2}+\|v\|_{H^2}^{r_j-2}\right)\|v\|_{H^2}$$
for all $v\in H^2(\R^N)$.
\end{lem}

\begin{proof}
Since
$$h_j(v)=J\int_{0}^{1}
\{Df_j(\varphi+\theta v)-Df_j(\varphi)\}v\,d\theta,$$
it follows from \eqref{ass:f} that
$$\|h_j(v)\|_{L^2}+\|h_j(v)\|_{L^{r_j'}}
\le C\left\{\begin{array}{ll}
\|v\|_{H^2}^{r_j-1} &
\hspace{2mm} \mbox{if} \hspace{3mm}
2<r_j\le 3,\\
(\|\varphi\|_{H^2}^{r_j-3}+\|v\|_{H^2}^{r_j-3})\|v\|_{H^2}^2 &
\hspace{2mm} \mbox{if} \hspace{3mm}
r_j>3,
\end{array}\right.$$
which implies the desired estimate.
\end{proof}

In what follows,
let $\lambda$ and $\mu$ be numbers satisfying
\begin{equation}\label{lam}
0<\lambda<\Re \lambda^*<\nu<\mu<(1+\alpha)\lambda,
\end{equation}
and we define
$$\|v\|_{X_T}=\|v\|_{L^{\infty,\lambda}_TH^2}
+\|\partial_tv\|_{L^{q_1,\lambda}_TL^{r_1}}
+\|\partial_tv\|_{L^{q_2,\lambda}_TL^{r_2}}.$$

\begin{lem}\label{lem22}
Let $v(t)$ be an $H^2$-solution of \eqref{eq:v} in $[0,\infty)$.
Then, there exists a constant $C$ independent of $v$ and $T$ such that
\begin{align*}
\|v\|_{X_T}\le C &
\left(\|v\|_{L^{\infty,\lambda}_TL^2}
+\|\partial_tv\|_{L^{\infty,\lambda}_TL^2}
+\|\partial_tv\|_{L^{q_1,\lambda}_TL^{r_1}}
+\|\partial_tv\|_{L^{q_2,\lambda}_TL^{r_2}}\right) \\
&+C\left(\|v\|_{X_T}^2+\|v\|_{X_T}^{r_1-1}+\|v\|_{X_T}^{r_2-1}\right).
\end{align*}
\end{lem}

\begin{proof}
By Lemma \ref{lem21}, we have
\begin{align*}
&\|v(t)\|_{H^2}\le C(\|v(t)\|_{L^2}+\|Av(t)\|_{L^2}) \\
&\le C(\|v(t)\|_{L^2}+\|\partial_tv(t)\|_{L^2}+\|h(v(t))\|_{L^2}) \\
&\le C(\|v(t)\|_{L^2}+\|\partial_tv(t)\|_{L^2}+\|v(t)\|_{H^2}^2
+\|v(t)\|_{H^2}^{r_1-1}+\|v(t)\|_{H^2}^{r_2-1})
\end{align*}
for all $t\in [0,T]$. Thus,
\begin{align*}
\|v\|_{L^{\infty,\lambda}_TH^2}
\le C & (\|v\|_{L^{\infty,\lambda}_TL^2}
+\|\partial_tv\|_{L^{\infty,\lambda}_TL^2}) \\
&+C(\|v\|_{L^{\infty,\lambda}_TH^2}^2
+\|v\|_{L^{\infty,\lambda}_TH^2}^{r_1-1}
+\|v\|_{L^{\infty,\lambda}_TH^2}^{r_2-1}),
\end{align*}
which implies the desired estimate.
\end{proof}

\begin{lem}\label{lem23}
There exists a constant independent of $v$ and $T$ such that
$$\|h_j(v)\|_{L^{q_j',\mu}_TL^{r_j'}}
\le C\left(\|v\|_{X_T}^2+\|v\|_{X_T}^{r_j-1}\right).$$
\end{lem}

\begin{proof}
By Lemma \ref{lem21}, we have
$$e^{-\mu t}\|h_j(v(t))\|_{L^{r_j'}}
\le Ce^{(2\lambda-\mu)t}\|e^{-}_{\lambda}v\|_{L_T^{\infty}H^2}^2
+Ce^{((r_j-1)\lambda-\mu)t}\|e^{-}_{\lambda}v\|_{L_T^{\infty}H^2}^{r_j-1}$$
for all $t\in [0,T]$.
Moreover, by \eqref{alp} and \eqref{lam}, we have
\begin{align*}
e^{\mu T}\|e^{-}_{\mu}h_j(v)\|_{L_T^{q_j'}L^{r_j'}}
&\le Ce^{2\lambda T}\|e^{-}_{\lambda}v\|_{L_T^{\infty}H^2}^2
+Ce^{(r_j-1)\lambda T}\|e^{-}_{\lambda}v\|_{L_T^{\infty}H^2}^{r_j-1} \\
&\le C(\|v\|_{X_T}^2+\|v\|_{X_T}^{r_j-1}),
\end{align*}
which implies the desired estimate.
\end{proof}

\begin{lem}\label{lem24}
There exists a constant $C$ independent of $v$ and $T$ such that
$$\|\partial_th_j(v)\|_{L^{q_j',\mu}_TL^{r_j'}}
\le C\left(\|v\|_{X_T}^2+\|v\|_{X_T}^{r_j-1}\right).$$
\end{lem}

\begin{proof}
Since
$\partial_th_j(v(t))=J\{Df_j(\varphi+v(t))-Df_j(\varphi)\}\partial_tv(t)$,
it follows from \eqref{ass:f} that
$$\|\partial_th_j(v(t))\|_{L^{r_j'}}
\le C(\|v(t)\|_{H^2}+\|v(t)\|_{H^2}^{r_j-2})\|\partial_tv(t)\|_{L^{r_j}}.$$
Thus we have
\begin{align*}
e^{-\mu t}\|\partial_th_j(v(t))\|_{L^{r_j'}}
\le & \, Ce^{(2\lambda-\mu) t}\|e^{-}_{\lambda}v\|_{L_T^{\infty}H^2}
\cdot e^{-\lambda t}\|\partial_tv(t)\|_{L^{r_j}} \\
&+Ce^{((r_j-1)\lambda-\mu) t}\|e^{-}_{\lambda}v\|_{L_T^{\infty}H^2}^{r_j-2}
\cdot e^{-\lambda t}\|\partial_tv(t)\|_{L^{r_j}}
\end{align*}
for all $t\in [0,T]$.
Moreover, by \eqref{alp}, \eqref{lam}
and the H\"older inequality,
\begin{align*}
&e^{\mu T}\|e^{-}_{\mu}\partial_th_j(v)\|_{L_T^{q_j'}L^{r_j'}} \\
\le & Ce^{2\lambda T}\|e^{-}_{\lambda}v\|_{L_T^{\infty}H^2}
\|e^{-}_{\lambda}\partial_tv\|_{L_T^{q_j}L^{r_j}}
+Ce^{(r_j-1)\lambda T}\|e^{-}_{\lambda}v\|_{L_T^{\infty}H^2}^{r_j-2}
\|e^{-}_{\lambda}\partial_tv\|_{L_T^{q_j}L^{r_j}} \\
\le & C(\|v\|_{X_T}^2+\|v\|_{X_T}^{r_j-1}).
\end{align*}
This completes the proof.
\end{proof}

\begin{proof}[Proof of Theorem \ref{mainthm}]
We use the argument in \cite[Section 6]{GSS2}
(see also \cite{CCO2,SS2}).
Suppose that the standing wave $e^{i\omega t}\varphi$
of \eqref{nls} is orbitally stable.
For small $\delta>0$, let $u_{\delta}(t)$ be the solution of
\eqref{nls} with $u_{\delta}(0)=\varphi+\delta \Re \chi$,
where $\chi\in H^2(\R^N,\C)^2$ is the eigenfunction of $A$
corresponding to the eigenvalue $\lambda^*$
given in Lemma \ref{lem:chi}.
Note that $A\overline{\chi}=\overline{\lambda^*}\overline{\chi}$.
Since either $\Re \chi\not\in \ker A$ or $\Im \chi\not\in \ker A$,
we may assume that $\Re \chi\not\in \ker A$.
Since we assume that $e^{i\omega t}\varphi$
is orbitally stable in $H^1(\R^N)$,
the $H^1$-solution $u_{\delta}(t)$ of \eqref{nls} exists globally
for sufficiently small $\delta>0$.
Moreover, since $\varphi$, $\chi\in H^2(\R^N)$,
by the $H^2$-regularity for \eqref{nls},
we see that $u_{\delta}\in
C([0,\infty),H^2(\R^N))\cap C^1([0,\infty),L^2(\R^N))$ and
$\partial_t u_{\delta}\in L^{q_1}_TL^{r_1}\cap L^{q_2}_TL^{r_2}$
for all $T>0$ (see \cite{kat,tsu} and also \cite[Section 5.2]{caz}).
By the change of variables
\begin{equation}\label{cv}
u_{\delta}(t)=e^{i\omega t}(\varphi+v_{\delta}(t)),
\end{equation}
we see that $v_{\delta}$ has the same regularity as that of $u_{\delta}$,
and satisfies
\begin{align}
\partial_tv_{\delta}(t)&=Av_{\delta}(t)+h(v_{\delta}(t)), \quad
v_{\delta}(0)=\delta \Re \chi, \nonumber \\
v_{\delta}(t)&=\delta \Re (e^{\lambda^*t}\chi)+\Gamma [h(v_{\delta})](t),
\label{eq:n2}\\
\partial_tv_{\delta}(t)&=\delta \Re (\lambda^* e^{\lambda^*t}\chi)
+e^{tA}h(\delta \Re \chi)+\Gamma [\partial_t h(v_{\delta})](t)
\label{eq:n3}
\end{align}
for all $t\ge 0$.
Let $\varepsilon_0>0$ be a small positive number to be determined later,
let $k=1$ if $\Im \lambda^*=0$, and
$k=\exp(2\pi \Re \lambda^*/|\Im \lambda^*|)$ if $\Im \lambda^*\ne 0$,
and define $T_{\delta}$ by
\begin{equation}\label{td}
\log \frac{\varepsilon_0}{k\delta}\le \Re \lambda^*T_{\delta}
\le \log \frac{\varepsilon_0}{\delta}, \quad
\Im \lambda^*T_{\delta}\in 2\pi \mathbb{Z}.
\end{equation}
for small $\delta>0$.
First, we prove that there exist constants $C_1$ and $\varepsilon_0$
independent of $\delta$ such that
\begin{equation}\label{be1}
\|v_{\delta}\|_{X_{T_{\delta}}}\le C_1\varepsilon_0
\end{equation}
for small $\delta$.
For $T\in (0,T_{\delta}]$, by \eqref{eq:n2},
Proposition \ref{prop2} and Lemma \ref{lem23},
\begin{align*}
\|v_{\delta}\|_{L^{\infty,\lambda}_{T}L^2}
&\le \|\delta e^{+}_{\lambda^*}\chi\|
_{L^{\infty,\lambda}_{T}L^2}
+C(\|h_1(v)\|_{L^{q_1',\mu}_{T}L^{r_1'}}
+\|h_2(v)\|_{L^{q_2',\mu}_{T}L^{r_2'}})  \\
&\le \delta e^{\Re \lambda^* T}\|\chi\|_{L^2}
+C(\|v_{\delta}\|_{X_{T}}^2
+\|v_{\delta}\|_{X_{T}}^{r_1-1}
+\|v_{\delta}\|_{X_{T}}^{r_2-1}).
\end{align*}
Moreover, by \eqref{eq:n3}, Proposition \ref{prop2}
and Lemma \ref{lem24},
\begin{align*}
&\|\partial_tv_{\delta}\|_{L^{\infty,\lambda}_TL^2}
+\|\partial_tv_{\delta}\|_{L^{q_1,\lambda}_TL^{r_1}}
+\|\partial_tv_{\delta}\|_{L^{q_2,\lambda}_TL^{r_2}} \\
&\le C\left(\delta e^{\Re \lambda^*T}\|\chi\|_{H^2}
+e^{\nu T}\|h(\delta \Re \chi)\|_{L^2}
+\|v_{\delta}\|_{X_{T}}^2
+\|v_{\delta}\|_{X_{T}}^{r_1-1}
+\|v_{\delta}\|_{X_{T}}^{r_2-1}\right).
\end{align*}
Here, by Lemma \ref{lem21} and by \eqref{alp} and \eqref{lam},
\begin{align*}
e^{\nu T}\|h(\delta \Re \chi)\|_{L^2}
&\le Ce^{\nu T} (\delta^2\|\chi\|_{H^2}^2
+\delta^{r_1-1}\|\chi\|_{H^2}^{r_1-1}
+\delta^{r_2-1}\|\chi\|_{H^2}^{r_2-1}) \\
&\le C(\delta e^{\Re \lambda^*T})^{1+\alpha}.
\end{align*}
Therefore, by Lemma \ref{lem22} and \eqref{td},
\begin{equation}\label{be2}
\|v_{\delta}\|_{X_{T}}
\le C\left(\varepsilon_0+\varepsilon_0^{1+\alpha}
+\|v_{\delta}\|_{X_{T}}^2
+\|v_{\delta}\|_{X_{T}}^{r_1-1}
+\|v_{\delta}\|_{X_{T}}^{r_2-1}\right)
\end{equation}
for all $T\in (0,T_{\delta}]$.
Since $\limsup_{T\to +0}\|v_{\delta}\|_{X_T}\le C\delta$
and $\|v_{\delta}\|_{X_{T}}$ is continuous in $T$,
by \eqref{be2} we see that
there exist constants $C_1$ and $\varepsilon_0$
independent of $\delta$ such that \eqref{be1}
holds for small $\delta$.
Next, by \eqref{eq:n2}, \eqref{be1},
Proposition \ref{prop2} and Lemma \ref{lem23},
\begin{align}
&\|v_{\delta}(T_{\delta})-\delta \Re (e^{\lambda^*T_{\delta}}\chi)\|_{L^2}
\le C(\|h_1(v)\|_{L^{q_1',\mu}_{T_{\delta}}L^{r_1'}}
+\|h_2(v)\|_{L^{q_2',\mu}_{T_{\delta}}L^{r_2'}})
\nonumber \\
&\le C(\|v_{\delta}\|_{X_{T_{\delta}}}^2
+\|v_{\delta}\|_{X_{T_{\delta}}}^{r_1-1}
+\|v_{\delta}\|_{X_{T_{\delta}}}^{r_2-1})
\le C\varepsilon_0^{1+\alpha}.
\label{be3}
\end{align}
Let $(\Re \chi)^{\perp}$ be the projection of $\Re \chi$
onto the orthogonal complement of
$\Span \{i\varphi,\nabla \varphi\}$ in $L^2(\R^N,\R)^2$.
Note that we identify
$i\varphi=(0,\varphi)$ and $\varphi=(\varphi,0)$.
Since $\Span \{i\varphi,\nabla \varphi\}\subset \ker A$
and $\Re \chi \not\in \ker A$,
we see that $(\Re \chi)^{\perp}\ne 0$.
By \eqref{td} and \eqref{be3}, we have
\begin{align*}
&\left|(v_{\delta}(T_{\delta}),(\Re \chi)^{\perp})_{L^2}
-\delta e^{\Re \lambda^*T_{\delta}}
\|(\Re \chi)^{\perp}\|_{L^2}^2\right| \\
&=|(v_{\delta}(T_{\delta})-\delta \Re (e^{\lambda^*T_{\delta}}\chi),
(\Re \chi)^{\perp})_{L^2}|
\le C\varepsilon_0^{1+\alpha}\|(\Re \chi)^{\perp}\|_{L^2},
\end{align*}
and we can take a small $\varepsilon_0$ such that
\begin{equation}\label{be4}
(v_{\delta}(T_{\delta}),(\Re \chi)^{\perp})_{L^2}
\ge \frac{\varepsilon_0}{2k}\|(\Re \chi)^{\perp}\|_{L^2}^2.
\end{equation}
Finally, we put
$$\Theta_{\delta}=\inf_{(\theta,y)\in \R\times \R^N}
\|u_{\delta}(T_{\delta})-e^{i\theta}\varphi(\cdot+y)\|_{L^2}.$$
Then, by \eqref{cv},
$\Theta_{\delta}=\inf_{(\theta,y)\in \R\times \R^N}
\|v_{\delta}(T_{\delta})+\varphi-e^{i\theta}\varphi(\cdot+y)\|_{L^2}$,
and there exists
$(\theta_{\delta},y_{\delta})\in \R\times \R^N$ such that
$\Theta_{\delta}=\|v_{\delta}(T_{\delta})+\varphi
-e^{i\theta_{\delta}}\varphi(\cdot+y_{\delta})\|_{L^2}$.
Moreover, since $\Theta_{\delta}\le \|v_{\delta}(T_{\delta})\|_{L^2}
\le C_1\varepsilon_0$,
we have $\|\varphi-e^{i\theta_{\delta}}\varphi(\cdot+y_{\delta})\|_{L^2}
\le 2C_1\varepsilon_0$.
Thus, $|(\theta_{\delta},y_{\delta})|=O(\varepsilon_0)$ and
$$e^{i\theta_{\delta}}\varphi(\cdot+y_{\delta})-\varphi
=i\theta_{\delta}\varphi+y_{\delta}\cdot \nabla \varphi+o(\varepsilon_0),$$
which together with \eqref{be4} implies that
\begin{align*}
&(v_{\delta}(T_{\delta})+\varphi
-e^{i\theta_{\delta}}\varphi(\cdot+y_{\delta}),
(\Re \chi)^{\perp})_{L^2} \\
&=(v_{\delta}(T_{\delta}), (\Re \chi)^{\perp})_{L^2}
-(i\theta_{\delta}\varphi+y_{\delta}\cdot \nabla \varphi,
(\Re \chi)^{\perp})_{L^2}-o(\varepsilon_0) \\
&\ge \frac{\varepsilon_0}{4k}\|(\Re \chi)^{\perp}\|_{L^2}^2
\end{align*}
for some small $\varepsilon_0$. Therefore,
$$\inf_{(\theta,y)\in \R\times \R^N}
\|u_{\delta}(T_{\delta})-e^{i\theta}\varphi(\cdot+y)\|_{H^1}
\ge \Theta_{\delta}\ge \frac{\varepsilon_0}{4k}
\|(\Re \chi)^{\perp}\|_{L^2}$$
for all $\delta$ small.
This contradiction proves that
$e^{i\omega t}\varphi$ is orbitally unstable.
\end{proof}

\section{Remark on spectral mapping theorem}\label{sect:SMT}

In this section, we assume that $V_{jk}\in C(\R^N,\R)$ and
there exist positive constants $\varepsilon$ and $C$ such that
\begin{equation}\label{exp-decay}
|V_{jk}(x)|\le C e^{-2\varepsilon |x|}
\end{equation}
for all $x\in \R^N$ and $j,k=1,2$.
We consider the linear operator $A=A_0+V$ defined by \eqref{AV}.
Then we have the following.

\begin{prop}\label{prop:SMT}
For each $N\ge 1$ one has $\sigma(e^{A})=e^{\sigma (A)}$.
\end{prop}

In \cite{GJLS}, Proposition \ref{prop:SMT}
is proved for the case $V_{11}=V_{22}=0$.
We modify the proof of Theorem 1 of \cite{GJLS}
to prove Proposition \ref{prop:SMT} for general case.
As we have stated in Section \ref{sect:intro},
this generalization is needed to treat the case where
a solution $\varphi$ of \eqref{sp} is not real-valued.

\begin{proof}[Proof of Proposition \ref{prop:SMT}]
For $\xi=a+i\tau$ with $a$, $\tau\in \R\setminus\{0\}$, we denote
$$L(\xi)=\left[\begin{array}{cc}
\xi & -D \\
D & \xi
\end{array}\right], \quad
D=-\Delta +\omega.$$
Then, we have $-\xi^2 \notin \sigma (D^2)$ and
$$L(\xi)^{-1}
=\left[\begin{array}{cc}
\xi [\xi^2+D^2]^{-1} & D[\xi^2+D^2]^{-1} \\
-D[\xi^2+D^2]^{-1} & \xi [\xi^2+D^2]^{-1}
\end{array}\right].$$
We also have
$\xi-A=L(\xi)-V=L(\xi)[I-L(\xi)^{-1}V]$.
Here we decompose $V=WB$ by
$$W=e^{\varepsilon |x|}V, \quad B=e^{-\varepsilon |x|}I.$$
By \eqref{exp-decay}, all entries of $W$ and $B$ are
exponentially decaying continuous functions.
Moreover, each entry of $BL(\xi)^{-1}W$ has a form
$$P_1(x)\xi [\xi^2+D^2]^{-1}Q_1(x)
+P_2(x)D [\xi^2+D^2]^{-1}Q_2(x),$$
where $P_1$, $P_2$, $Q_1$ and $Q_2$ are real-valued
continuous functions decaying exponentially.
Therefore, by Lemma 6 of \cite{GJLS}, we see that
$\|BL(\xi)^{-1}W\|\to 0$ as $|\tau|\to \infty$.
Then the rest of the proof of Proposition \ref{prop:SMT}
is the same as in the proof of Theorem 1 of \cite{GJLS}.
\end{proof}

\vspace{3mm}

\textbf{Acknowledgements.}
This work started when V.G. visited Saitama University
by the RIMS International Project Research 2009
\lq\lq Qualitative study on nonlinear partial differential equations
of dispersive type".
Part of this work was done while M.O. visited Universit\'e Bordeaux 1
by the JSPS Excellent Young Researchers Overseas Visit Program.
M.O. would like to thank Mathieu Colin and Thierry Colin
for their hospitality. V.G. was supported by the Italian National Council of Scientific Research (project PRIN No.
2008BLM8BB )
entitled: "Analisi nello spazio delle fasi per E.D.P."

\end{document}